\documentclass[psamsfonts]{amsart}

\usepackage{amssymb,amsfonts}
\usepackage[all,arc]{xy}
\usepackage{enumerate}
\usepackage{enumitem}
\usepackage{mathrsfs}
\usepackage{hyperref}
\usepackage{cleveref}
\usepackage{graphicx}
\usepackage{mathtools}
\usepackage[dvipsnames]{xcolor}
\usepackage[T1]{fontenc}
\usepackage[utf8]{inputenc}

\newtheorem{thm}{Theorem}[section]
\newtheorem{cor}[thm]{Corollary}
\newtheorem{prop}[thm]{Proposition}
\newtheorem{lem}[thm]{Lemma}

\newtheorem{obs}[thm]{Observation}

\theoremstyle{definition}
\newtheorem{defn}[thm]{Definition}

\newtheorem{exmp}[thm]{Example}

\newtheorem{fact}[thm]{Fact}

\theoremstyle{remark}
\newtheorem{rem}[thm]{Remark}

\makeatletter
\let\c@equation\c@thm
\makeatother
\numberwithin{equation}{section}

\usepackage{graphicx}

\title[Zariski density, critical exponents, and unitary representations]{Zariski density of discrete subgroups via critical exponents and unitary representations}
\author{Aleksander Skenderi}
\address{Department of Mathematics, University of Wisconsin-Madison}
\email{askenderi@wisc.edu}
\begin{document}
\begin{abstract}
Let $G$ be a connected real semisimple linear algebraic group with finite center and no compact factors, let $X$ denote its associated Riemannian symmetric space, and let $h_{\mathrm{vol}}(X)$ be the volume growth entropy of $X$. We show that there exists an $\epsilon = \epsilon(G) > 0$ so that the following holds: if $\Gamma < G$ is a discrete subgroup with critical exponent greater than $h_{\mathrm{vol}}(X) - \epsilon$, then $\Gamma$ is Zariski dense in $G$. If $G$ is further assumed to be isomorphic to one of the isometry groups of the real, complex, or quaternionic hyperbolic spaces, we determine the smallest possible value of the critical exponent to guarantee Zariski density of the associated subgroup (in other words, the largest possible value of $\epsilon = \epsilon(G)$). In the setting of discrete subgroups of real semisimple Lie groups with no compact factors, this generalizes Borel's Density Theorem for lattices.
\par 
The key ingredient is input from the theory of unitary representations, particularly the recent work of Benoist--Liang (which was inspired by earlier work of Benoist--Kobayashi) on general temperedness criteria for the quasi-regular representation of $L^2(G/H)$ for any closed subgroup $H$ of $G$.
\end{abstract}
\maketitle
\tableofcontents
\section{Introduction}
\subsection{Motivation and statement of the main theorem}
Let $G$ be a connected real semisimple linear algebraic group without compact factors and with finite center. Recall that a \emph{lattice} in $G$ is a discrete subgroup $\Gamma$ of $G$ such that the quotient $G/\Gamma$ admits a finite $G$-invariant measure. Equivalently, a lattice corresponds to the fundamental group of a finite volume locally symmetric space whose universal cover is isometric to the symmetric space $X$ of $G$. One of the first foundational results in the study of lattices in Lie groups is the famous density theorem of Armand Borel (\cite{Bor}, see also Theorem 3.2.5 of \cite{Zim}).
\begin{thm} [Borel's Density Theorem \cite{Bor}]
Let $G$ be a connected real semisimple algebraic group without compact factors. Let $\Gamma < G$ be a (topologically) closed subgroup so that $G/\Gamma$ admits a finite $G$-invariant measure. Then $\Gamma$ is Zariski dense in $G$.
\end{thm}
Lattices in Lie groups, particular of higher rank, are by now quite well-understood thanks to the celebrated superrigidity and arithmeticity theorems of Margulis \cite{Mar} as well as the works of many other mathematicians on the geometric and topological properties of their associated locally symmetric spaces; see for instance \cite{ABBGNRS}, \cite{FHR1}, \cite{FHR2}, \cite{FMW}, to name a few such works. As a result, a major focus of work in recent years has been on the study of infinite covolume discrete subgroups of Lie groups, with a particular emphasis of many mathematicians being the Anosov groups, as well as their numerous generalizations (such as relatively Anosov groups and transverse groups, among others).
\par 
However, very few general results are known about \emph{arbitrary} infinite covolume discrete subgroups of Lie groups, with a classification of such subgroups (if at all possible), being seemingly far from reach. Indeed, to the best of the author's knowledge, the only such result is the recent spectacular breakthrough of Frączyk--Gelander \cite{FG}, which states the following: if $G$ is a higher rank semisimple Lie group with Kazhdan's property (T) and $\Gamma < G$ is discrete, then the locally symmetric space $\Gamma \backslash X = \Gamma \backslash G / K$ has infinite volume if and only if it has infinite injectivity radius.
\par 
Given how fundamental Borel's Density Theorem is in the study of lattices in Lie groups, the author was led to wonder to what extent one can expect a generalization of this theorem to hold in the setting of infinite covolume discrete subgroups. This is precisely the motivation behind the present work. Of course, there are infinite covolume discrete subgroups which are not Zariski dense. Therefore we are in search of a general criterion to guarantee when a discrete subgroup is Zariski dense, which, in particular, recovers Borel's renowned result in the case when the discrete subgroup is a lattice. Our criterion will involve the \emph{critical exponent} of $\Gamma$. 
\par 
Let $G$ be as above, let $K < G$ be a maximal compact subgroup of $G$, and let $X := G/K$ denote the symmetric space of $G$. Fix a $G$-invariant metric $d_{X}$ induced from a Riemannian metric on $X$; for instance, (see section \ref{Liegroupsandunitaryreps}) fixing a $K$-invariant norm $|| \cdot ||$ on $\mathfrak{a}$ induced from the Killing form, define $d_{X}(gK, hK) = ||\kappa(g^{-1}h)||$ for all $g,h \in G$, where $\kappa : G \rightarrow \mathfrak{a}^{+}$ denotes the Cartan projection. Fix a basepoint $o \in X$. The critical exponent of $\Gamma$, denoted by $\delta(\Gamma)$, is the abscissa of convergence of the \emph{Poincar\'e series}
\begin{align*}
Q_{\Gamma}(s) := \sum_{\gamma \in \Gamma} e^{-sd_{X}(o,\gamma o)},
\end{align*} that is,
\begin{align*}
\delta(\Gamma) := \inf \{s > 0 : Q_{\Gamma}(s) < \infty \} \in [0, \infty].
\end{align*} Equivalently, the critical exponent of $\Gamma$ is given by 
\begin{align*}
\delta(\Gamma) = \limsup_{T \to \infty} \frac{1}{T} \log \# \{ \gamma \in \Gamma : d_{X}(o, \gamma o) \leq T \},
\end{align*} and thus it measures the exponential growth rate of the $\Gamma$-orbits in $X$. Since $\Gamma$ acts on $X$ by isometries, the critical exponent is independent of the choice of basepoint $o \in X$, hence is well-defined. When $\Gamma$ is a lattice in $G$, its critical exponent coincides with the \emph{volume growth entropy of} $X$ (see \cite{Leu}), which is defined by 
\begin{align*}
h_{\mathrm{vol}}(X) := \lim_{R \to \infty} \frac{1}{R} \log \mathrm{vol}(B_{R}(o)).
\end{align*} Here $B_{R}(o)$ is the ball of radius $R$ centered at $o \in X$ and $\mathrm{vol}$ denotes the Riemannian volume on $X$. We remark that the volume growth entropy is independent of the choice of basepoint $o \in X$. If $\Gamma < G$ is an arbitrary discrete subgroup, then $\delta(\Gamma) \leq h_{\mathrm{vol}}(X)$ and understanding the values $\delta(\Gamma)$ can take is a question which has attracted much interest; see for instance \cite{Cor}, \cite{DL}, \cite{Leu}, \cite{Q2}, and \cite{S1}. We will come back to this point later. We are now in a position to state our main result.
\begin{thm} [Theorem \ref{mainthm2}]
\label{mainthm1}
Let $G$ be a connected real semisimple linear algebraic group without compact factors and with finite center. There exists a constant $\epsilon = \epsilon(G) > 0$ so that the following holds: if $\Gamma < G$ is a discrete subgroup such that
\begin{align*}
\delta(\Gamma) > h_{\mathrm{vol}}(X) - \epsilon,
\end{align*} then $\Gamma$ is Zariski dense in $G$.
\end{thm}
When $G$ has real rank one, we are able to give a more precise statement; see Theorem \ref{realrankoneversion}. Let us make a couple of comments on the content of this theorem.
\begin{rem}
As mentioned previously, when $\Gamma < G$ is a lattice, we have $\delta(\Gamma) = h_{\mathrm{vol}}(X)$, hence the hypothesis of the theorem is immediately satisfied and we recover, with a new proof, Borel's Density Theorem (for lattices).
\end{rem}
\begin{rem}
When $G$ has Kazhdan's property (T), that is, when $G$ has no simple factors locally isomorphic to SO$(n,1)$ or SU$(n,1)$, the Corlette--Leuzinger gap phenomenon \cite{Cor}, \cite{Leu} (see also related work of Quint \cite{Q2}) states that there exists $\epsilon_{0} = \epsilon_{0}(G) > 0$ (which is not necessarily the $\epsilon$ of Theorem \ref{mainthm1}! See the discussion following Theorem \ref{realrankoneversion}.) so that any infinite covolume discrete subgroup $\Gamma$ of $G$ has critical exponent $\delta(\Gamma) \leq h_{\mathrm{vol}}(X) - \epsilon_{0}$. In other words, in the property (T) case only lattices have critical exponent larger than $h_{\mathrm{vol}}(X) - \epsilon_{0}$, and so in this setting a stronger conclusion than that of Theorem \ref{mainthm1} is true. We emphasize however that the proof of our result does $\textbf{not employ}$ the theorems of Corlette and Leuzinger. Indeed, the proof of Theorem \ref{mainthm1} applies at once both to groups with and without property (T). We do however essentially rely on spectral gap arguments and tools from the theory of unitary representations, as we now explain.
\end{rem}
\subsection{Brief sketch of the proof}
\label{ideasofproof}
The idea of the proof is to show that there is some uniform constant $M < h_{\mathrm{vol}}(X)$, depending only on the ambient Lie group $G$, so that, if $\Gamma < G$ is \emph{not} Zariski dense, then $\delta(\Gamma) \leq M$. If $\Gamma$ is not Zariski dense, then it is contained in a proper closed connected Lie subgroup $H$ of $G$. If $H$ is reductive, then it preserves a proper totally geodesic Riemannian symmetric subspace $Y$ of $X$, and therefore $\delta(\Gamma) \leq h_{\mathrm{vol}}(Y) < h_{\mathrm{vol}}(X)$. But if $H$ has a non-trivial unipotent radical, then it may act in quite a complicated way on the symmetric space $X$. In this case however, $H$ is a subgroup of a \emph{proper} parabolic subgroup $Q$ of $G$, and the idea now is to use the recent work of Benoist--Liang \cite{BL}, as well as a spectral gap result inspired by work of Einsiedler--Margulis--Venkatesh \cite{EMV} and Moore \cite{Mo} (see Proposition \ref{semisimplealmostL2m}), to show that there is a uniform upper bound on the critical exponents of discrete subgroups of proper parabolic subgroups of $G$.
\subsection{Organization of the paper}
In section \ref{Liegroupsandunitaryreps}, we recall some necessary background on the structure theory of semisimple Lie groups, as well as some important facts from the theory of unitary representations. In section \ref{BackgrounonBenoist--Liang}, we briefly survey the recent work of Benoist--Liang \cite{BL}, primarily focusing on their results that will be needed in the proof of our main theorem. In section \ref{establishingmainthm}, we prove Theorem \ref{mainthm1}, as well as a more quantitative result in real rank one, and discuss some other quick applications of the main result.
\subsection*{Acknowledgements}
I want to thank my advisors Andrew Zimmer and Sebastian Hurtado-Salazar, as well as Mikołaj Frączyk, for many helpful discussions and suggestions. I am especially grateful to Mikołaj for his generous help. His numerous suggestions and answers to my questions were very important at several stages of this work. In particular, Mikołaj suggested I look into the work of Benoist--Kobayashi (this later led me to study the paper of Benoist--Liang, which plays an essential role in the present paper), and answered several of my questions about unitary representations. 
\par 
This material is based upon work supported by the National Science Foundation under Grant No. DMS-2230900.

\section{Semisimple Lie groups and their unitary representations}
\label{Liegroupsandunitaryreps}
\subsection{Basic structure theory of semisimple Lie groups}
Let $G$ be a connected real semisimple linear algebraic group without compact factors and with finite center and let $\mathfrak{g}$ denote its Lie algebra. Let $b$ denote the Killing form of $\mathfrak{g}$ and fix a Cartan involution $\tau$ of $\mathfrak{g}$; that is, an involution of $\mathfrak{g}$ for which the bilinear pairing $\langle \cdot, \cdot \rangle$ on $\mathfrak{g}$ defined by $\langle X, Y \rangle := -b(X, \tau (Y))$ is an inner product. Then $\mathfrak{g}$ decomposes as $\mathfrak{g} = \mathfrak{k} \oplus \mathfrak{p}$, where $\mathfrak{k}$ and $\mathfrak{p}$ are the $1$ and $-1$ eigenspaces of $\tau$. The subalgebra $\mathfrak{k}$ is a maximal compact Lie subalgebra of $\mathfrak{g}$ and we denote by $K < G$ the maximal compact Lie subgroup of $G$ whose Lie algebra is $\mathfrak{k}$.

Fix a maximal abelian subspace $\mathfrak{a} \subset \mathfrak{p}$, known as a \emph{Cartan subspace}, which is unique up to conjugation. The Lie algebra $\mathfrak{g}$ then decomposes as 
\begin{align*}
    \mathfrak{g} = \mathfrak{g}_{0} \oplus \bigoplus_{\alpha \in \Sigma} \mathfrak{g}_{\alpha},
\end{align*} which is called the \emph{restricted root space decomposition} associated to $\mathfrak{a}$; in this decomposition, for $\alpha \in \mathfrak{a}^{*}$, we define
\begin{align*}
    \mathfrak{g}_{\alpha} := \{X \in \mathfrak{g} : [H,X] = \alpha (H)X \ \ \mathrm{for \ all} \ H \in \mathfrak{a} \},
\end{align*} and call
\begin{align*}
    \Sigma := \{ \alpha \in \mathfrak{a}^{*} \smallsetminus \{0\} : \mathfrak{g}_{\alpha} \neq 0 \}
\end{align*} the set of \emph{restricted roots}. Now fix an element $H_{0} \in \mathfrak{a}$ so that $\alpha (H_{0}) \neq 0$ for all $\alpha \in \Sigma$, and let
\begin{align*}
\Sigma^{+} := \{ \alpha \in \Sigma : \alpha (H_{0}) > 0 \} \ \ \mathrm{and} \ \ \Sigma^{-} := - \Sigma^{+}.
\end{align*} Notice that $\Sigma = \Sigma^{+} \sqcup \Sigma^{-}$. We write $\Delta \subset \Sigma^{+}$ for the set of \emph{simple restricted roots}, which, by definition, consists of all the elements of $\Sigma^{+}$ which cannot be written as a non-trivial positive integer linear combination of elements in $\Sigma^{+}$. As $\Sigma$ is an abstract root system on $\mathfrak{a}^{*}$, it follows that $\Delta$ is a basis of $\mathfrak{a}^{*}$ and every $\alpha \in \Sigma^{+}$ is a non-negative integral linear combination of elements in $\Delta$. See for instance Chapter II of \cite{Kn1} for more details.
\subsubsection{The Weyl group, Cartan projection, and Jordan projection}
 The \emph{Weyl group} of $\mathfrak{a}$ is given by $\mathcal{W} := N_{K}(\mathfrak{a}) / Z_{K}(\mathfrak{a})$, where $N_{K}(\mathfrak{a}) \subset K$ is the normalizer of $\mathfrak{a}$ in $K$ and $Z_{K}(\mathfrak{a}) \subset K$ is the centralizer of $\mathfrak{a}$ in $K$. The Weyl group is a finite group generated by reflections of $\mathfrak{a}$ (with respect to the inner product $\langle \cdot, \cdot \rangle$) about the kernels of the simple restricted roots in $\Delta$. Hence, $\mathcal{W}$ acts transitively on the set of \emph{Weyl chambers}, which are the closures of the connected components of 
\begin{align*}
    \mathfrak{a} - \bigcup_{\alpha \in \Sigma} \ker \alpha.
\end{align*} We call the Weyl chamber 
\begin{align*}
    \mathfrak{a}^{+} := \{X \in \mathfrak{a} : \alpha (X) \geq 0 \ \ \mathrm{for \ all} \ \alpha \in \Delta \},
\end{align*} the \emph{positive Weyl chamber}. We set $A := \exp (\mathfrak{a})$, $A^{+} := \exp (\mathfrak{a}^{+})$, and $\mathfrak{a}^{++} = \mathrm{int}(\mathfrak{a}^{+})$. Let $\kappa : G \rightarrow \mathfrak{a}^{+}$ denote the \emph{Cartan projection}, that is $\kappa (g) \in \mathfrak{a}^{+}$ is the unique element so that 
\begin{align*}
    g = k \exp (\kappa (g)) \ell
\end{align*} for some $k, \ell \in K$. We note that $k, \ell \in K$ need not be unique. Such a decomposition of $g \in G$ is called a $KA^{+}K$ decomposition (see Theorem $7.39$ of \cite{Kn1}). Using the Cartan projection, we can define the map $\lambda : G \rightarrow \mathfrak{a}^{+}$ known as the \emph{Jordan projection} by 
\begin{align*}
\lambda(g) := \lim_{n \to \infty} \frac{\kappa(g^{n})}{n}.
\end{align*} An element $g \in G$ is said to be $\emph{loxodromic}$ if $\lambda(g) \in \mathfrak{a}^{++}$. We denote by 
\begin{align*}
\rho = \frac{1}{2} \sum_{\alpha \in \Sigma^{+}} (\dim \mathfrak{g}_{\alpha}) \alpha \in \mathfrak{a}^{*},
\end{align*} the half-sum of the positive roots, counted with multiplicity, which appears frequently in the study of decay properties of unitary representations. 
\subsubsection{Parabolic subgroups and the Levi decomposition}
\label{parabolicandlevi}
Nice references for the material in this section include section 5 of Chapter V of \cite{Kn2} and section 30 of \cite{Hum}.
\par 
The normalizer in $G$ of the nilpotent subalgebra
\begin{align*}
\mathfrak{n} = \bigoplus_{\alpha \in \Sigma^{+}} \mathfrak{g}_{\alpha}
\end{align*} is the \emph{standard minimal parabolic subgroup}, denoted by $P$. By a \emph{minimal parabolic subgroup} we mean a conjugate of $P$. A subgroup $Q < G$ is called a \emph{parabolic subgroup} if it contains a minimal parabolic subgroup. There are finitely many conjugacy classes of parabolic subgroups of $G$.
\par 
Consider now an \emph{arbitrary} connected real algebraic group $H$. Then $H$ admits a \emph{Levi decomposition}
\begin{align*}
H = L \ltimes R_{u}(H),
\end{align*} where $L$ is a reductive subgroup of $H$ known as the \emph{Levi subgroup} and $R_{u}(H)$ is the unipotent radical of $H$. If the unipotent radical of $H$ is trivial, then $H$ is reductive. Otherwise $H$ is contained in a parabolic subgroup of $G$ (see Theorem 30.4 (a) of \cite{Hum}).

\subsection{Some ingredients from the theory of unitary representations}
The general references for this section include \cite{BdV} and \cite{Kn2}.
\par 
Unless otherwise specified, we will assume in this section that $G$ is a locally compact second countable topological group. A \emph{unitary representation} of $G$ is a pair $(\pi, \mathcal{H})$ consisting of a complex Hilbert space $\mathcal{H}$ and a group homomorphism $\pi : G \rightarrow \mathcal{U}(\mathcal{H})$ from $G$ to the group $\mathcal{U}(\mathcal{H})$ of unitary operators of $\mathcal{H}$ which is \emph{strongly continuous} in the sense that, for any $v \in \mathcal{H}$, the map $G \rightarrow \mathcal{H}$ given by $g \mapsto \pi(g)v$ is continuous. A \emph{matrix coefficient} of $\pi$ is a map $\phi_{v_{1}, v_{2}} : G \rightarrow \mathbb{C}$ given by
\begin{align*}
\phi_{v_{1}, v_{2}}(g) := \langle \pi(g)v_{1}, v_{2} \rangle,
\end{align*} where $v_{1}, v_{2} \in \mathcal{H}$. The strong continuity of $\pi$ ensures that matrix coefficients are bounded continuous functions on $G$.
\par 
\begin{exmp} [Left regular representation and quasi-regular representation]
The prototypical example of a unitary representation of $G$ is the so-called \emph{left regular representation} $\lambda_{G}$ of $G$ on the space $L^2(G)$ of (equivalence classes of) square-integrable functions on $G$. Let $dx$ denote a left Haar measure on $G$ and let $\langle \cdot, \cdot \rangle$ denote the usual $L^2$ scalar product of complex-valued functions on $G$:
\begin{align*}
\langle f_{1}, f_{2} \rangle := \int_{G} f_{1}(x) \overline{f_{2}(x)} dx.
\end{align*} The left regular representation is then the unitary representation $\lambda_{G} : G \rightarrow \mathcal{U}(L^2(G))$ given by
\begin{align}
\label{lefttranslation}
\lambda_{G}(g)f : x \mapsto f(g^{-1}x), \ \ \mathrm{for} \ \ f \in L^2(G).
\end{align} If $H < G$ is a closed subgroup of $G$, there is a similarly defined representation of $G$ on the space $L^2(G/H)$, denoted $\lambda_{G/H}$, and called the \emph{quasi-regular representation}. Note that, while the homogeneous space $G/H$ always admits a $G$-quasi-invariant Radon measure (see Lemma B.1.3 of \cite{BdV}), it might not admit a $G$-invariant measure (this occurs if and only if the restriction of the modular function of $G$ to $H$ coincides with the modular function of $H$). Therefore it is not sufficient to simply perform left translation by group elements of $G$, as in (\ref{lefttranslation}); one also needs to include a ``scaling factor,'' coming from the Radon--Nikodym derivative of the push-forward of the quasi-invariant measure with respect to itself, to ensure that one indeed obtains a unitary representation.

We remark that all the standard concepts to be introduced (temperedness, being almost $L^p$, etc.) are independent of the choice of the $G$-quasi-invariant Radon measure.
\end{exmp}
\begin{defn} [Weak Containment]
Let $(\sigma, \mathcal{V})$ and $(\pi, \mathcal{H})$ be two unitary representations of $G$. We say that $\sigma$ is \emph{weakly contained in} $\pi$ (written $\sigma \prec \pi$) if every diagonal matrix coefficient $\langle \sigma(\cdot) v, v \rangle$, $v \in \mathcal{V}$, can be approximated uniformly on compact sets by convex combinations of diagonal matrix coefficients of $\pi$.
\end{defn}
Intuitively, if $\sigma \prec \pi$ and if we have information on the decay of matrix coefficients of $\pi$, then we have information on the decay of matrix coefficients of $\sigma$. Since many spectral properties of unitary representations are phrased in terms of decay of matrix coefficients, weak containment is an extremely useful property when one has spectral information about one unitary representation and wants to obtain such information about another unitary representation.
\par 
Now let $G$ be a noncompact real semisimple algebraic group. The following definition, due to Harish-Chandra, plays a fundamental role in the study of harmonic analysis on semisimple Lie groups.
\begin{defn} [Tempered representation]
A unitary representation $(\pi, \mathcal{H})$ of $G$ is called \emph{tempered} if it is weakly contained in the left regular representation $(\lambda_{G}, L^2(G))$ of $G$.
\end{defn}
\begin{fact}
\label{tempforminparabolic}
Let $G$ be a noncompact real semisimple algebraic group and let $P < G$ be a minimal parabolic subgroup. Then the representation $(\lambda_{G/P}, L^2(G/P))$ of $G$ is tempered. This is well-known but since it will be useful in what follows (see Theorem \ref{realrankoneversion}), we roughly explain why this is the case. For more details, see for instance sections 3.1 and 3.2 of \cite{BL}.
\par 
Recall that $P$, being a minimal parabolic subgroup, is amenable. This is equivalent to the trivial representation $\mathbf{1}_{P}$ of $P$ being weakly contained in the regular representation of $P$. A process known as (unitary) induction preserves weak containment. Inducing the trivial representation $\mathbf{1}_{P}$ from $P$ to $G$ yields the quasi-regular representation $\lambda_{G/P}$ of $G$, while inducing the regular representation $\lambda_{P}$ of $P$ gives the regular representation $\lambda_{G}$ of $G$. Hence $\lambda_{G/P} \prec \lambda_{G}$, as desired.
\end{fact} Temperedness can be defined in many equivalent ways. One of them, which will be especially useful for us, is in terms of the optimal integrability of matrix coefficients. First we need another definition. We caution the reader that the following terms are widely used in the literature, although their precise meaning depends on the authors.
\begin{defn}
Given $p \in [1, \infty)$, a unitary representation $(\pi, \mathcal{H})$ of $G$ is said to be:
\begin{itemize}
    \item[(a)] \emph{strongly} $L^p$ if there is a dense subspace $\mathcal{D} \subset \mathcal{H}$ so that $\langle \pi(\cdot)v, v \rangle \in L^p(G)$ for all $v \in \mathcal{D}$;
    \item[(b)] \emph{almost} $L^p$ if it is strongly $L^{p+\epsilon}$ for all $\epsilon > 0$;
    \item[(c)] \emph{totally} $L^p$ if all matrix coefficients of $\pi$ are in $L^{p}$; 
    \item[(d)] \emph{totally} $L^{p+}$ if it is totally $L^{p+\epsilon}$ for all $\epsilon > 0$. 
\end{itemize} A priori, item (b) is weaker than item (d). A result of Samei--Wiersma provides conditions on when they are in fact equivalent.
\begin{lem} [Theorem 5.3 of \cite{SW}]
\label{almosttotallyequiv}
For a noncompact real semisimple algebraic group $G$ and $2 \leq p < \infty$, a unitary representation is almost $L^p$ if and only if it is totally $L^{p+}$.
\end{lem}
\end{defn}
A famous result of Cowling--Haagerup--Howe allows us to express temperedness in terms of being almost $L^2$. Precisely:
\begin{thm}[Theorem 1 of \cite{CHH}]
\label{CowlingHaagerupHowe}
Let $G$ be a real semisimple algebraic group. For any unitary representation $(\pi, \mathcal{H})$ of $G$, the following are equivalent:
\begin{itemize}
    \item[(i)] $\pi$ is tempered.
    \item[(ii)] $\pi$ is almost $L^2$.
\end{itemize}
\end{thm}
\begin{rem}
Cowling--Haagerup--Howe in fact also characterize temperedness by giving an upper-bound on $K$-finite matrix coefficients in terms of the Harish-Chandra function $\Xi = \Xi_{G}$ of $G$ (here $K$ is a maximal compact subgroup of $G$). Since this will not play a role in the paper, we decided to omit the precise formulation of this result.
\end{rem}
\subsection{Tensor Product of Unitary Representations}
We now introduce the tensor product of unitary representations, which will appear in the proof of Proposition \ref{semisimplealmostL2m}. Let $(\mathcal{H}, \langle \cdot, \cdot \rangle_{\mathcal{H}})$ and $(\mathcal{V},  \langle \cdot, \cdot \rangle_{\mathcal{V}})$ be Hilbert spaces, and $\mathcal{H} \otimes \mathcal{V}$ their vector space tensor product. The completion of $\mathcal{H} \otimes \mathcal{V}$, with respect to the unique inner product $\langle \cdot, \cdot \rangle$ for which
\begin{align*}
\langle v_{1} \otimes w_{1}, v_{2} \otimes w_{2} \rangle = \langle v_{1}, v_{2} \rangle_{\mathcal{H}} \cdot \langle w_{1}, w_{2} \rangle_{\mathcal{V}}, \ \  \mathrm{for \ all} \ \ v_{1}, v_{2} \in \mathcal{H}, \ w_{1}, w_{2} \in \mathcal{V}
\end{align*} is called the $\emph{Hilbert tensor product}$ of $\mathcal{H}$ and $\mathcal{V}$, and is again denoted by $\mathcal{H} \otimes \mathcal{V}$.
\begin{defn} [Tensor product of unitary representations]
If $(\pi, \mathcal{H})$ and $(\sigma, \mathcal{V})$ are unitary representations of $G$, their $\emph{tensor product}$ is the unitary representation $\pi \otimes \sigma$ of $G$ defined on $\mathcal{H} \otimes \mathcal{V}$ by
\begin{align*}
(\pi \otimes \sigma)(g)(v \otimes w) := \pi(g)v \otimes \sigma(g)w,
\end{align*} for all $v \in \mathcal{H}, w \in \mathcal{V}$, and $g \in G$.
\end{defn} One similarly defines the tensor product of finitely many unitary representations $\{(\pi_{i}, \mathcal{H}_{i})\}_{1 \leq i \leq k}$, writing $\big( \otimes_{i=1}^{k} \pi_{i} , \otimes_{i=1}^{k} \mathcal{H}_{i} \big)$ for the resulting unitary representation. We record the following immediate, but very useful, observation regarding tensor products.
\begin{obs}
\label{densityofproductofcoeffs}
If $\{(\pi_{i}, \mathcal{H}_{i})\}_{1 \leq i \leq k}$ is a finite collection of unitary representations of a topological group $G$, then the set consisting of the products of their matrix coefficients is dense in the set of matrix coefficients of $\big( \otimes_{i=1}^{k} \pi_{i} , \otimes_{i=1}^{k} \mathcal{H}_{i} \big)$. 
\end{obs}
\section{Background on the work of Benoist--Liang}
\label{BackgrounonBenoist--Liang}
As discussed in the previous section, tempered representations are essential in the study of harmonic analysis on semisimple Lie groups. In a series of papers, \cite{BK1}, \cite{BK2}, \cite{BK3}, \cite{BK4}, Benoist--Kobayashi conducted a very comprehensive study to determine necessary and sufficient geometric criteria for when the quasi-regular representation $(\lambda_{G/H}, L^2(G/H))$ of $G$ on a homogeneous space $G/H$ is tempered. In \cite{BL}, Benoist--Liang unified both the results of this series of works, as well as work of Lutsko--Weich--Wolf \cite{LWW}, by considering the action of $G$ on $G/H$ for \emph{any closed subgroup} $H$ of $G$ and providing temperedness criteria for $(\lambda_{G/H}, L^2(G/H))$ in terms of several quantites associated to this action. We will recall the definitions of three of the four so-called ``exponents'' of Benoist--Liang, as well as some of their results, which will be essential in section \ref{establishingmainthm}.
\begin{defn} [Definition 1.2 of \cite{BL}]
The \emph{optimal integrability exponent} $p_{G/H}$ is defined as the infimum of all $p \in [1, \infty]$ so that, for any $f \in C_{c}(G/H)$, we have 
\begin{align*}
\langle \lambda_{G/H}(\cdot) f, f \rangle \in L^{p}(G).
\end{align*}
\end{defn}
\emph{Motivation of Benoist--Liang for} $p_{G/H}$. By Theorem \ref{CowlingHaagerupHowe}, we have
\begin{align}
\label{equiv1}
L^2(G/H) \ \mathrm{is \ tempered} \ \Longleftrightarrow \ p_{G/H} \leq 2.
\end{align}
\begin{defn} [Definition 1.1 of \cite{BL}]
The \emph{coefficient decay exponent} $\theta_{G/H}$ is defined as the infimum of $\theta \in [0,1]$ such that, for any $f \in C_{c}(G/H)$, there exists a constant $C = C(f) > 0$ so that, uniformly for all $g \in G$, we have
\begin{align*}
| \langle \lambda_{G/H}(g)f,f \rangle | \leq C e^{2(\theta - 1) \rho(\kappa(g))}.  
\end{align*}
\end{defn} 
\emph{Motivation of Benoist--Liang for} $\theta_{G/H}$. By Theorem 1 and 2 of \cite{CHH}, we have
\begin{align}
\label{equiv2}
L^2(G/H) \ \mathrm{is \ tempered} \ \Longleftrightarrow \ \theta_{G/H} \leq 1/2.
\end{align} See Corollary 3.10 of \cite{BL} for more details.
\begin{defn} [Definition 1.3 of \cite{BL}]
The \emph{relative volume growth exponent} $\delta_{G/H}$ of $H$ inside $G$, which turns out to lie in $[0,1]$ is defined by
\begin{align*}
\delta_{G/H} := \max \bigg \{0, \limsup_{g_{n} \rightarrow \infty} \frac{\log \nu_{H}(H \cap Bg_{n}B)}{\log \nu_{G}(Bg_{n}B)} \bigg \},
\end{align*} where $B$ is any compact subset of $G$ of non-empty interior and the notation ``$g_{n} \to \infty$'' means that the sequence $\{g_{n}\}$ eventually escapes every compact subset of $G$. The measure $\nu_{G}$ is the Haar measure on $G$, while $d\nu_{H}(h) = (\det \mathrm{Ad}_{H}h)^{1/2} dh$ is the symmetric measure on $H$ (see (2.1) of \cite{BL}).
\end{defn}
\emph{Motivation of Benoist--Liang for} $\delta_{G/H}$. In the case when $H = \Gamma$ is a discrete subgroup of $G$, the exponent $\delta_{G/H} = \delta_{G/\Gamma}$ is closely related to several measurements of the asymptotic growth rate of $\Gamma$-orbits in $X = G/K$ which are already present in the literature.
\par 
For a discrete subgroup $\Gamma < G$, recall that \emph{Quint's growth indicator function} $\psi_{\Gamma} : \mathfrak{a}^{+} \rightarrow \mathbb{R} \cup \{-\infty\}$, introduced by Quint in \cite{Q1}, is defined as
\begin{align*}
\psi_{\Gamma}(v) = ||v|| \inf_{\mathcal{C} \ni v} \inf \bigg \{ s \in \mathbb{R} : \sum_{\gamma \in \Gamma, \kappa(\gamma) \in \mathcal{C}} e^{-s ||\kappa(\gamma)||} < \infty \bigg \},
\end{align*} where the infimum is taken over all open cones $\mathcal{C} \subset \mathfrak{a}^{+}$  containing $v$. The quantity on the right-hand side is independent on the choice of norm $||\cdot||$ on $\mathfrak{a}^{+}$, hence the growth indicator function is well-defined. Benoist--Liang established the following relation between $\delta_{G/\Gamma}$ and $\psi_{\Gamma}$.
\begin{prop} [Proposition 4.17 of \cite{BL}]
\label{relationwithgrowthind}
If $\Gamma$ is a discrete subgroup of $G$, then 
\begin{align*}
\delta_{G/\Gamma} = \max \bigg \{ \sup_{v \in \mathfrak{a}^{+}} \frac{\psi_{\Gamma}(v)}{2 \rho(v)}, 0 \bigg \}.
\end{align*} 
\end{prop}
From the definition of the growth indicator function, one sees that it is a homogeneous function which satisfies $\psi_{\Gamma}(v) \leq \delta(\Gamma)$ for all unit vectors $v \in \mathfrak{a}^{+}$. When $G$ has real rank one, $\psi_{\Gamma}$ coincides with $\delta(\Gamma)$. Furthermore, we always have $2||\rho|| = h_{\mathrm{vol}}(X)$ and, when $G$ has rank one, $2 \rho = h_{\mathrm{vol}}(X)$ (see for instance \cite{Leu}). Hence, if $G$ has rank one and $\Gamma < G$ is a discrete subgroup, we have 
\begin{align}
\label{realrankonerelations}
\delta(\Gamma) = h_{\mathrm{vol}}(X) \cdot \delta_{G/\Gamma}.
\end{align} To conclude this expository section, we recall several important facts about these exponents, established by Benoist--Liang, which will be crucial in what follows.
\begin{lem} [Lemma 6.14 of \cite{BL}]
\label{comparisonofcoeffdecay}
Let $G$ be a noncompact real semisimple algebraic group. For any closed subgroups $H_{2} < H_{1} < G$, we have $\theta_{G/H_{2}} \leq \theta_{G/H_{1}}$.
\end{lem} 
\begin{thm} [Theorem A of \cite{BL}]
\label{equalityofexponents}
Let $H$ be a closed subgroup of $G$. Then,
\begin{align*}
\theta_{G/H} = \delta_{G/H} = 1 - \frac{1}{p_{G/H}}.
\end{align*}
\end{thm} Notice that Lemma \ref{comparisonofcoeffdecay} and the first equality of Theorem \ref{equalityofexponents} immediately give:
\begin{cor}
\label{comparisonofcritexp}
For any closed subgroups $H_{2} < H_{1} < G$, we have $\delta_{G/H_{2}} \leq \delta_{G/H_{1}}$.
\end{cor} Finally, by Theorem \ref{equalityofexponents} and either of (\ref{equiv1}) or (\ref{equiv2}), we have
\begin{align}
\label{equiv3}
L^2(G/H) \ \mathrm{is \ tempered} \ \Longleftrightarrow \ \delta_{G/H} \leq \frac{1}{2}.
\end{align}
\section{Establishing the main theorem}
\label{establishingmainthm}
Let $G$ be a connected real semisimple linear algebraic group without compact factors and with finite center, and let $\Gamma < G$ be a discrete subgroup. Note that replacing $G$ by $G/Z(G)$ and $\Gamma$ by its projection onto $G/Z(G)$ changes neither the symmetric space $X$ of $G$ nor the critical exponent of $\Gamma$ for its action on $X$. Hence there is no loss of generality in assuming that $Z(G) = \{\mathrm{id}\}$. Then $G$ decomposes as a direct product
\begin{align*}
G = \prod_{i=1}^{k} G_{i},
\end{align*} where the $G_{i}$ are the simple factors of $G$. The purpose of making this reduction is that it will allow us to employ a spectral gap result of Einsiedler--Margulis--Venkatesh \cite{EMV} and Moore \cite{Mo} in the proof of Proposition \ref{semisimplealmostL2m} below. Recall first that a connected real linear algebraic group is simple if and only if it is almost simple (meaning that all of its proper normal subgroups are either finite or of finite index). Then by Lemma 6.5 of \cite{EMV} and Theorem 3.5 and Proposition 3.6 of Moore's work \cite{Mo} (see also item (6.1) on page 164 of \cite{EMV} and the discussion preceeding it), we have the following result:
\begin{prop} [Einsiedler--Margulis--Venkatesh, Moore]
\label{simplealmostL2m}
Let $G$ be a connected simple non-compact real algebraic group and let $H < G$ be a proper closed connected subgroup. Then the quasi-regular representation $\lambda_{G/H}$ of $G$ on $L^2(G/H)$ is almost $L^{2m}$ for some $1 \leq m < \infty$.
\end{prop}
With this result at hand, we obtain the following proposition, which is a key technical ingredient in our proof of Theorem \ref{mainthm1}.
\begin{prop}
\label{semisimplealmostL2m}
Let $G$ be a connected real semisimple linear algebraic group without compact factors and with finite center and let $Q < G$ be a proper parabolic subgroup of $G$. Then the quasi-regular representation $\lambda_{G/Q}$ of $G$ on $L^2(G/Q)$ is almost $L^{2m}$ for some $1 \leq m < \infty$.
\end{prop}
\begin{proof}
As discussed in the beginning of the section, we may assume that the center of $G$ is trivial, and therefore we can write $G = \prod_{i=1}^{k} G_{i}$, where the $G_{i}$ are the connected non-compact simple factors of $G$. Let $Q < G$ be a proper parabolic subgroup. Notice that $Q$ itself decomposes as a product $Q = \prod_{i=1}^{k} Q_{i}$, with $Q_{i}$ a parabolic subgroup of $G_{i}$ for all $1 \leq i \leq k$. Furthermore, the set
\begin{align*}
I := \{1 \leq i \leq k : Q_{i} \ \mathrm{is \ a \ proper \ subgroup \ of} \ G_{i} \}
\end{align*} is non-empty, as otherwise we would have $Q = G$. If $i \notin I$, then the quasi-regular representation $\lambda_{G_{i} / Q_{i}}$ is simply the one-dimensional trivial unitary representation of the Hilbert space $\mathbb{C}$. Therefore we have the following isomorphisms of unitary representations:
\begin{align*}
\lambda_{G/Q} \cong \bigotimes_{i=1}^{k} \lambda_{G_{i}/Q_{i}} \cong \bigotimes_{i \in I} \lambda_{G_{i}/Q_{i}}.
\end{align*}
By Proposition \ref{simplealmostL2m}, we know that, for each $i \in I$, there exists some $1 \leq m_{i} < \infty$ so that $\lambda_{G_{i}/Q_{i}}$ is almost $L^{2m_{i}}$. By Lemma \ref{almosttotallyequiv}, we in fact have the (a priori) stronger conclusion that $\lambda_{G_{i}/Q_{i}}$ is totally $L^{2m_{i}+}$. Now fix $q \geq \max \{2m_{i} : i \in I\}$ sufficiently large so that 
\begin{align*}
m := \frac{q}{2 \# I} \geq 1.
\end{align*} By definition, for each $i \in I$, all the matrix coefficients of $\lambda_{G_{i}/Q_{i}}$ are in $L^{q}(G)$. Now let $\mathcal{M} \subset L^{2}(G/Q)$ denote the subset consisting of products of the matrix coefficients of the unitary representations $\big \{ \big(\lambda_{G_{i}/Q_{i}}, L^{2}(G_{i}/Q_{i}) \big)\big \}_{i \in I}$. By construction,
\begin{align*}
\sum_{i \in I} \frac{1}{q} = \frac{\# I}{q} = \frac{1}{2m},
\end{align*} hence the generalized H\"older inequality implies that $\phi \in L^{2m}(G)$, for all $\phi \in \mathcal{M}$. But by Observation \ref{densityofproductofcoeffs}, $\mathcal{M}$ is dense in $L^{2}(G/Q)$. Hence $\lambda_{G/Q}$ is strongly $L^{2m}$, so in particular almost $L^{2m}$.

\end{proof}
Lastly we record the following elementary but useful lemma.
\begin{lem}
\label{invariantunderconjugation}
Let $\Gamma < G$ be a discrete subgroup and let $g \in G$ be arbitrary. Then
\begin{align*}
\delta(g \Gamma g^{-1}) = \delta (\Gamma).
\end{align*}
\end{lem}
\begin{proof}
By symmetry, it suffices to show that $\delta(g \Gamma g^{-1}) \geq \delta (\Gamma)$. Let $R > 0$ be arbitrary. If $\gamma \in \Gamma$ satisfies $d_{X}(o, \gamma o) \leq R$ then, since $G$ acts on $X$ by isometries, we have
\begin{align*}
d_{X}(o, g \gamma g^{-1} o) &\leq d_{X}(o, go) + d_{X}(go, g \gamma o) + d_{X}(g \gamma o, g \gamma g^{-1} o) \\
&= d_{X}(o, go) + d_{X}(o, \gamma o) + d_{X}(o, g^{-1} o) \\
&\leq R+2C,
\end{align*} where $C := d_{X}(o, go)$. It follows that
\begin{align*}
\# \{\eta \in g \Gamma g^{-1} : d_{X}(o, \eta o) \leq R+2C \} \geq \# \{ \gamma \in \Gamma : d_{X}(o, \gamma o) \leq R \}
\end{align*} for all $R > 0$. Thus
\begin{align*}
\delta (g \Gamma g^{-1}) &= \limsup_{R \to \infty} \frac{1}{R+2C} \log \# \{\eta \in g \Gamma g^{-1} : d_{X}(o, \eta o) \leq R+2C \} \\
&\geq \limsup_{R \to \infty} \frac{R}{R+2C} \cdot \frac{1}{R} \log \# \{ \gamma \in \Gamma : d_{X}(o, \gamma o) \leq R \} \\
&= \delta(\Gamma),
\end{align*} as desired.
\end{proof} 
We can now prove our main result, Theorem \ref{mainthm1}, which we restate below.
\begin{thm} [Theorem \ref{mainthm1}]
\label{mainthm2}
Let $G$ be a connected real semisimple linear algebraic group without compact factors and with finite center. There exists a constant $\epsilon = \epsilon(G) > 0$ so that the following holds: if $\Gamma < G$ is a discrete subgroup such that
\begin{align*}
\delta(\Gamma) > h_{\mathrm{vol}}(X) - \epsilon,
\end{align*} then $\Gamma$ is Zariski dense in $G$.
\end{thm}
\begin{proof}
Let $\mathcal{S}$ denote the set of proper totally geodesic Riemannian symmetric subspaces of $X$ and let $\mathcal{Q}$ denote the set of proper parabolic subgroups $Q$ of $G$ which contain the standard minimal parabolic subgroup $P$. Define
\begin{align*}
M := \max \bigg \{ \Big \{h_{\mathrm{vol}}(Y) \ : \ Y \in \mathcal{S} \Big \} \bigcup \Big \{ \Big(1 - \frac{1}{p_{G/Q}} \Big)h_{\mathrm{vol}}(X) \ : \ Q \in \mathcal{Q} \Big \} \bigg \}.
\end{align*}
We first claim that $M < h_{\mathrm{vol}}(X)$. Indeed, if two Riemannian manifolds are isometric, then they have the same volume growth entropies; since there are finitely many equivalence classes of isometric Riemannian symmetric subspaces of $X$ and each proper symmetric subspace has strictly smaller volume growth entropy, we see that 
\begin{align}
\label{firstbound}
\max \{h_{\mathrm{vol}}(Y) : Y \in \mathcal{S}\} < h_{\mathrm{vol}}(X).
\end{align}
Furthermore, there are finitely many proper parabolic subgroups $Q \in \mathcal{Q}$, and for each such subgroup Proposition \ref{semisimplealmostL2m} furnishes an integer $1 \leq m_{Q} < \infty$ so that $\lambda_{G/Q}$ is almost $L^{2m_{Q}}$. By Lemma \ref{almosttotallyequiv}, this is the same as $\lambda_{G/Q}$ being totally $L^{2m_{Q}+}$. Therefore we have $p_{G/Q} \leq 2m_{Q} < \infty$, hence
\begin{align*}
1 - \frac{1}{p_{G/Q}} \leq 1 - \frac{1}{2m_{Q}} < 1.
\end{align*}
Since $\mathcal{Q}$ is a finite set, we have
\begin{align}
\label{secondbound}
\max \bigg \{ \Big(1 - \frac{1}{p_{G/Q}} \Big) h_{\mathrm{vol}}(X) : Q \in \mathcal{Q} \bigg \} < h_{\mathrm{vol}}(X).
\end{align}
Combining (\ref{firstbound}) and (\ref{secondbound}), we obtain $M < h_{\mathrm{vol}}(X)$. 
\par 
To prove the theorem, it suffices to show that if $\Gamma < G$ is discrete subgroup which is $\textbf{not}$ Zariski dense, then $\delta(\Gamma) \leq M$. This will show that the theorem holds with $\epsilon = \epsilon(G) = h_{\mathrm{vol}}(X) - M > 0$. So assume that $\Gamma < G$ is not Zariski dense. Then $\Gamma$ is contained in a \emph{proper} closed connected subgroup $H$ of $G$. Let
\begin{align*}
H = L \ltimes R_{u}(H),
\end{align*} be the Levi decomposition of $H$, where $L$ is the Levi factor and $R_{u}(H)$ is the unipotent radical of $H$. There are two cases to consider:
\par 
$\textbf{Case 1:}$ Suppose first that the unipotent radical $R_{u}(H)$ is trivial, i.e., $R_{u}(H) = \{\mathrm{id}\}$. Then $H$ is a proper reductive subgroup of $G$, and therefore preserves a proper totally geodesic Riemannian symmetric subspace $Y$ of $X$. Since $\Gamma$ is contained in $H$, the maximal possible value of its critical exponent is that of a lattice in $H$, which is precisely the volume growth entropy of $Y$. Hence in this case $\delta(\Gamma) \leq h_{\mathrm{vol}}(Y) \leq M$, as desired.
\par 
$\textbf{Case 2:}$ If the unipotent radical is not trivial, then, as was discussed in section \ref{parabolicandlevi}, $H$ is contained in a proper parabolic subgroup $Q$ of $G$. By Lemma \ref{invariantunderconjugation}, conjugation by elements of $G$ does not change the critical exponent of $\Gamma$. Hence there is no loss of generality in assuming that $Q$ contains the standard minimal parabolic subgroup $P$, that is, $Q \in \mathcal{Q}$. Combining Proposition \ref{relationwithgrowthind}, Theorem \ref{equalityofexponents}, and Corollary \ref{comparisonofcritexp}, we obtain
\begin{align*}
\sup_{v \in \mathfrak{a}^{+}} \frac{\psi_{\Gamma}(v)}{2 \rho (v)} \leq \delta_{G/\Gamma} \leq \delta_{G/Q} = 1 - \frac{1}{p_{G/Q}},
\end{align*} that is
\begin{align}
\label{boundongrowthind}
\psi_{\Gamma}(v) \leq \Big(1 - \frac{1}{p_{G/Q}} \Big) 2 \rho(v)
\end{align} for all $v \in \mathfrak{a}^{+}$. Quint showed that there exists a unit vector $w \in \mathfrak{a}^{+}$ so that $\psi_{\Gamma}(w) = \delta(\Gamma)$ (page 2 of \cite{Q1}; for more details, see Proposition 3.5 of \cite{S1}). Since $2 \rho(w) \leq h_{\mathrm{vol}}(X)$, (\ref{boundongrowthind}) gives
\begin{align*}
\delta(\Gamma) = \psi_{\Gamma}(w) \leq \Big(1 - \frac{1}{p_{G/Q}} \Big) 2 \rho(w) \leq \Big(1 - \frac{1}{p_{G/Q}} \Big) h_{\mathrm{vol}}(X) \leq M.
\end{align*} This concludes the proof.
\end{proof}
Along with work of Benoist, this result immediately implies important point-set topological properties of the \emph{Benoist limit cone} of a discrete subgroup $\Gamma < G$. We first recall the definition of this fundamental object. In what follows, $\Gamma_{\mathrm{lox}}$ denotes the set of loxodromic elements of $\Gamma$.
\begin{defn} [Benoist's Limit Cone \cite{Ben}]
\label{limitcone}
The \emph{limit cone} of $\Gamma$ is the smallest closed cone $\mathcal{L}_{\Gamma}$ in $\mathfrak{a}^{+}$ containing $\lambda(\Gamma_{\mathrm{lox}})$. In other words, $\mathcal{L}_{\Gamma}$ is the closure of the union of the half-lines spanned by the Jordan projections of the loxodromic elements of $\Gamma$:
\begin{align*}
\mathcal{L}_{\Gamma} := \overline{\bigcup_{g \in \Gamma_{\mathrm{lox}}} \mathbb{R}^{+} \lambda(g)}.
\end{align*}
\end{defn} Note that the word \emph{cone} does not presuppose that $\mathcal{L}_{\Gamma}$ is convex, nor that it has non-empty interior. When $\Gamma$ is Zariski dense, these properties do indeed hold, this being a deep theorem of Benoist (Theorem 1.2 of \cite{Ben}). Thus Benoist's theorem and Theorem \ref{mainthm2} immediately imply the following:
\begin{cor}
\label{limitconecor}
Let $G$ and $\epsilon = \epsilon(G) > 0$ be as in Theorem \ref{mainthm2}. If $\Gamma < G$ is a discrete subgroup such that $\delta(\Gamma) > h_{\mathrm{vol}}(X) - \epsilon$, then the limit cone $\mathcal{L}_{\Gamma}$ is a closed convex cone of non-empty interior.
\end{cor}
Informally, this says that knowing that a discrete subgroup has sufficiently large growth in a \emph{single} direction in the positive Weyl chamber (namely, the direction corresponding to $\delta(\Gamma)$) is enough to conclude point-set topological information about all the directions in the positive Weyl chamber which are asymptotically approached by the $\Gamma$-orbits in the symmetric space $X$.
\par 
To conclude this section, we give a more quantitative version of Theorem \ref{mainthm2} in the case when $G$ has real rank one. For algebraic groups $G_{1}$ and $G_{2}$, we will write $G_{1} \simeq G_{2}$ to mean that $G_{1}$ and $G_{2}$ are \emph{locally isomorphic}, that is, they have isomorphic Lie algebras. 
\begin{thm}
\label{realrankoneversion}
Let $G$ be a connected noncompact simple real algebraic group of real rank one and let $\Gamma < G$ be a discrete subgroup. Then $\Gamma$ is Zariski dense in $G$ provided:
\begin{itemize}
    \item[(a)] $G \simeq \mathrm{SO}(n,1)$ and $\delta(\Gamma) > n-2$.
    \item[(b)] $G \simeq \mathrm{SU}(n,1)$ and $\delta(\Gamma) > 2n-2$.
    \item[(c)] $G \simeq \mathrm{Sp}(n,1)$ and $\delta(\Gamma) > 4n-2$.
    \item[(d)] $G \simeq \mathrm{F}_{4}^{-20}$ and $\delta(\Gamma) > 11$.
\end{itemize}
\end{thm}
\begin{proof}
The argument is the same as that of Theorem \ref{mainthm2}, but it simplifies considerably because, as $G$ has real rank one, there is a single conjugacy class of parabolic subgroups, namely that of the standard minimal parabolic subgroup $P < G$. As discussed in Fact \ref{tempforminparabolic}, the quasi-regular representation $\lambda_{G/P}$ is always tempered in this case, hence $\delta_{G/P} \leq \frac{1}{2}$. Thus if $\Gamma < G$ is a discrete subgroup which is not Zariski dense, its critical exponent is bounded above by the maximum of $h_{\mathrm{vol}}(X)/2$ and $\max_{Y \subsetneq X} h_{\mathrm{vol}}(Y)$, where $Y$ ranges over all proper totally geodesic Riemannian symmetric subspaces of $X$. The result then follows.
\end{proof}
\begin{rem}
We remark that each of the lower bounds in items (a), (b), and (c) is optimal. Indeed, in item (a) for instance, consider a Lie subgroup $H$ of $G \simeq \mathrm{SO}(n,1)$ which is isomorphic to SO$(n-1,1)$, and consider some lattice $\Gamma$ in $H$. Then $\Gamma$ is a discrete subgroup of $G$ which is not Zariski dense and has critical exponent $\delta(\Gamma) = h_{\mathrm{vol}}(\mathbb{H}_{\mathbb{R}}^{n-1}) = n-2$. Similar arguments work in cases $(b)$ and $(c)$.
\par 
However, I do not know whether the lower bound in item (d) is optimal. Indeed, any proper totally geodesic Riemannian symmetric subspace of $\mathbb{H}_{\mathbb{O}}^2$ is isometric to either $\mathbb{H}_{\mathbb{R}}^{8}$ or to $\mathbb{H}_{\mathbb{H}}^{2}$ (see section 6 of \cite{Cor}), whence 
\begin{align*}
\max_{Y \subsetneq \mathbb{H}_{\mathbb{O}}^{2}} h_{\mathrm{vol}}(Y) = \max \Big \{h_{\mathrm{vol}}(\mathbb{H}_{\mathbb{R}}^{8}), h_{\mathrm{vol}}(\mathbb{H}_{\mathbb{H}}^{2}) \Big \} = 10,
\end{align*} while $h_{\mathrm{vol}}(\mathbb{H}_{\mathbb{O}}^2)/2 = 11$. Thus asking whether the lower bound of $11$ for $\delta(\Gamma)$ is optimal is equivalent to asking whether there are discrete subgroups of the minimal parabolic subgroup $P$ of $\mathrm{F}_{4}^{-20}$ whose critical exponents are arbitrarily close (or equal) to $11$. I am unaware of the answer to this question.
\end{rem}
We conclude with a discussion of the relationship between the above result and Corlette's gap theorem. We first recall the statement of Corlette's result.
\begin{thm} [Corlette, Theorem 4.4 of \cite{Cor}]
\label{CorletteGapThm} ~\
\begin{itemize}
    \item[(1)] If $\Gamma < \mathrm{Sp}(n,1)$, $n \geq 2$, is a discrete subgroup, then $\delta(\Gamma) = 4n+2$ or $\delta(\Gamma) \leq 4n$. Moreover, $\delta(\Gamma) = 4n+2$ if and only if $\Gamma$ is a lattice.
    \item[(2)] If $\Gamma < \mathrm{F}_{4}^{-20}$ is a discrete subgroup, then $\delta(\Gamma) = 22$ or $\delta(\Gamma) \leq 16$. Moreover, $\delta(\Gamma) = 22$ if and only if $\Gamma$ is a lattice. 
\end{itemize}
\end{thm}
Corlette also showed that there exist infinite covolume discrete subgroups of Sp$(n,1)$ and $\mathrm{F}_{4}^{-20}$ whose critical exponents are greater than $4n-2$ and $10$, respectively (see Theorem $6.1$ of \cite{Cor}). Thus Theorem \ref{realrankoneversion} implies that these non-lattice discrete subgroups of Sp$(n,1)$ constructed by Corlette are in fact Zariski dense, while again the case of $\mathrm{F}_{4}^{-20}$ is more mysterious.

\end{document}